\documentclass[12pt]{amsart}
\usepackage{amssymb}
\usepackage{multicol}
\usepackage{colordvi}
\usepackage{graphicx}
\usepackage{epsfig, epsf}
\numberwithin{equation}{section}
\newtheorem{theorem}{Theorem}[section]
\newtheorem{proposition}[theorem]{Proposition}
\newtheorem{lemma}[theorem]{Lemma}
\newtheorem{corollary}[theorem]{Corollary}
\newtheorem{defn}[theorem]{Definition}
\newtheorem{ex}[theorem]{Example}
\newtheorem{rem}[theorem]{Remark}
\newenvironment{example}{\begin{ex}\rm}{\end{ex}}
\newenvironment{definition}{\begin{defn}\rm}{\end{defn}}

\copyrightinfo{}{}
\newcounter{FNC}[page]
\def\fauxfootnote#1{{\addtocounter{FNC}{2}\Magenta{$^\fnsymbol{FNC}$}%
     \let\thefootnote\relax\footnotetext{\Magenta{$^\fnsymbol{FNC}$#1}}}}

\newcommand{\spn}{\operatorname{span}}
\newcommand{\A}[1]{\operatorname{AGL}(#1,\Z)}
\newcommand{\GL}{\operatorname{GL}}

\newcommand{\F}{\mathbb{F}}

\newcommand{\R}{\mathbb{R}}
\newcommand{\Z}{\mathbb{Z}}
\newcommand{\N}{\mathbb{N}}
\renewcommand{\P}{\mathbb{P}}

\title{Minkowski Length of 3D Lattice Polytopes}

\author{Olivia Beckwith}
\address{Harvey Mudd College\\
Department of Mathematics\\
301 Platt Boulevard, Claremont, CA 91711\\
USA
}
\email{obeckwith@gmail.com}
\thanks{}

\author{Matthew Grimm}
\address{Department of Mathematics\\
        UCSD\\
        9500 Gilman Dr., \#0112\\
        La Jolla, CA 92093
        USA}
\email{mgrimm2@kent.edu}
\thanks{}

\author{Jenya Soprunova}
\address{Department of Mathematics\\
        Kent State University\\
        Summit Street, Kent, OH 44242\\
        USA}
\email{soprunova@math.kent.edu}
\urladdr{http://www.math.kent.edu/~soprunova/}
\author{Bradley Weaver}
\address{Grove City College\\
Department of Mathematics\\
100 Campus Drive,Grove City PA 16127\\
USA
}
\email{weaverbr1@gcc.edu}
\thanks{The authors were supported by  funds from NSF-REU Grant DMS-0755318}


\begin{document}

\begin{abstract}
We study the Minkowski length $L(P)$ of a
lattice polytope $P$, which is defined to be the largest number of
non-trivial primitive segments whose Minkowski sum lies in $P$. The
Minkowski length represents the largest possible number of factors in a
factorization of polynomials with exponent vectors in $P$, and shows up in
lower bounds for the minimum distance of toric codes.  In this paper we give a
polytime algorithm for computing $L(P)$ where $P$ is a 3D lattice
polytope.  

We next study 3D lattice polytopes of Minkowski length 1. In particular, we show that
if $Q$, a subpolytope of   $P$, is the Minkowski sum of $L=L(P)$ lattice polytopes $Q_i$,
each of Minkowski length 1,  then the total number of interior lattice points of the polytopes
$Q_1,\cdots, Q_L$ is at most 4. Both results extend previously known results for lattice polygons.
 Our methods differ substantially from those used in the two-dimensional case.
\end{abstract}

\maketitle

\section*{Introduction}

Let $P$ be a convex lattice polytope in $\R^n$. Then $P$ defines $\mathcal{L}_{\F}(P)$, a vector space  over a field $\F$ spanned by the monomials in $P$. That is,

$$\mathcal{L}_{\F}(P)
=\spn_{\F}\{t^m\ |\ m\in P\cap\Z^n\},$$
where $t^m=t_1^{m_1}\cdots t_n^{m_n}.$ In this paper we address the following question:  What is the largest number of factors that a polynomial in $\mathcal{L}_{\F}(P)$ could have?
We also study those factors and obtain results regarding their Newton polytopes.  Although these questions are interesting on its own,
our motivation comes from studying toric codes.

The {\it toric code} $\mathcal{C}_P$, first introduced by Hansen in
\cite{Han}, is defined by evaluating the polynomials in $\mathcal{L}_{\F_q}(P)$ at all the points
$t$ in the algebraic torus $(\F_q^*)^n$. That is, $\mathcal{C}_P$ is a linear code whose codewords 
are the strings $(f(t)\ |\ t\in (\F_q^*)^n)$ for
 $f\in\mathcal{L}_{\F_q}(P)$.
 It is convenient to assume that $P$ is contained in the square $[0,q-2]^n$,
so that all the monomials in $\mathcal{L}_{\F_q}(P)$ are linearly independent
over $\F_q$~\cite{Ru}. Thus $\mathcal{C}_P$ has block length  $(q-1)^n$ and 
dimension equal to the number of the lattice points in $P$.

Note that the weight of each non-zero codeword in $\mathcal{C}_P$ is
the number of points $t\in(\F_q^*)^n$ where the corresponding polynomial does not
vanish. Therefore, the minimum distance of $\mathcal{C}_P$ (which is the minimum weight for linear codes)
equals 
$$d(\mathcal{C}_P)=(q-1)^n-\max_{0\neq f\in\mathcal{L}_{\F_q}(P)}Z(f),$$
where $Z(f)$ is the number of zeroes  (i.e. points of vanishing) in $(\F_q^*)^n$ 
of $f$.

For toric surface codes, that is, in the case $n=2$, Little and Schenck in \cite{LSch} used Hasse-Weyl bound and the intersection theory on toric surfaces to come up with the following general idea: If $q$ is sufficiently large, then polynomials $f\in\mathcal{L}_{\F_q}(P)$ with more absolutely irreducible factors will
necessarily have more zeroes in $(\F_q^*)^2$ (\cite{LSch}, Proposition 5.2).  In \cite{SS1} this idea was expanded  to produce explicit bounds for the
minimum distance of $\mathcal{C}_P$ in terms of  a certain geometric invariant $L(P)$, (full) Minkowski length of $P$, which was introduced  in that paper.

This invariant $L(P)$ reflects the largest possible number of
absolutely irreducible factors a polynomial $f\in\mathcal{L}_{\F_q}(P)$ can have. A polytime algorithm for computing $L(P)$  for polygons was provided in \cite{SS1}.
In this paper,  we extend this result to dimension 3 (Theorem~\ref{T:alg}, based on Theorem~\ref{T:3d2}).

 Moreover, \cite{SS1} provides a description of the factorization $f=f_1\cdots f_{L(P)}$
for $f\in\mathcal{L}_{\F_q}(P)$ with the largest number of factors: it turns out that in such a factorization
the Newton polygon $P_{f_i}$ (which is the convex hull of the exponents of the
monomials in $f_i$) is either a primitive segment, a unit simplex, or
a triangle with exactly 1 interior point. It is also shown in \cite{SS1} that a triangle with an interior point
can occur in such a factorization at most once. This implies that the total number $I$ of interior 
lattice points of  $P_{f_i}$ is at most 1. This result is essential for establishing bounds on the minimum distance of toric surface codes in~\cite{SS1}.

This argument is not directly extendable to dimension 3, as  it does not seem feasible to obtain  a description of the Newton polytopes $P_{f_i}$ in dimension 3 (See~\cite{MSRI} for some examples).
Nevertheless, in this paper we show that in the 3D case $I \leq 4$ (Theorem~\ref{T:main}). Our methods
differ substantially from those used in the two-dimensional case. An initial version of our argument relied heavily on the  classification of 3D Fano tetrahedra \cite{Kas}. Although we were able to completely get rid of this dependency in our final argument, the classification helped us significantly in our explorations. 

In \cite{SS1} combinatorial results about Minkowski length of polygons  lead to lower bounds on the minimum distance of surface toric codes. We hope that our (entirely combinatorial) paper will in the future lead to similar bounds for 3D toric codes.

%
\section{Minkowski Length of Lattice Polytopes}

Here we recall the definition of the (full) Minkowski length introduced in \cite{SS1} as well as reproduce and refine some results
from that paper using new methods which will later be applied to the 3D case. 
\subsection{Minkowski sum}
Let $P$ and $Q$ be convex polytopes in $\R^n$. Their {\it Minkowski sum} is
$$P+Q=\{p+q\in\R^n\ |\ p\in P,\ q\in Q\},$$
which is again a convex polytope. Figure \ref{F:sum} shows the Minkowski sum of a triangle and a square.
\begin{figure}[h]
\centerline{
 \scalebox{0.55}
 {
\input{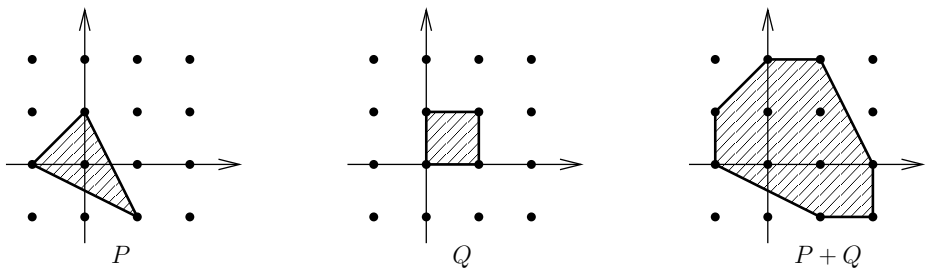}}}
\caption{The Minkowski sum of two polygons}
\label{F:sum}
\end{figure}

Let $f$ be a Laurent polynomial in  $\F_q[t_1^{\pm 1},\dots,t_n^{\pm 1}]$. Then its
{\it Newton polytope} $P_f$ is the convex hull of the exponent vectors 
of the monomials appearing in $f$. Thus $P_f$ is a {\it lattice polytope} as its vertices belong to the integer lattice $\Z^n\subset\R^n$. Note that if $f,g\in \F_q[t_1^{\pm 1},\dots,t_n^{\pm 1}]$ then the Newton polytope of their product $P_{fg}$ is the Minkowski sum $P_f+P_g$.  A {\it primitive segment} $E$
is a lattice segment whose only lattice points are its endpoints. The difference of the endpoints 
is a vector $v_E$ whose coordinates are relatively prime ($v_E$ is defined up to sign).
A polytope which is the Minkowski sum of primitive segments is called a {\it (lattice) zonotope}. 
We say that two lattice polytopes are {\it equal} if they are the same up to translation.

The automorphism group of the lattice is the group of {affine unimodular transformations}, denoted 
by  $\A{n}$, which consists of translations by an integer vector and linear transformations in $\GL(n,\Z)$. It is a standard fact from lattice-point geometry
that any two primitive segments in $\R^n$ are $\A{n}$-equivalent (\cite{Newman}, Theorem II.1).
\subsection{Minkowski length}
Let $P$ be a lattice polytope in $\R^n$.  

\begin{definition}\label{D:maximal}
The (full) Minkowski length $L=L(P)$ of a lattice polytope $P$ is the largest number
of primitive segments $E_1, E_2,\dots, E_L$ whose Minkowski sum is in $P$.
Equivalently, $L=L(P)$ is the largest number of non-trivial lattice polytopes $Q_1,\dots, Q_L$  whose 
Minkowski sum is in $P$.  Any collection of $L=L(P)$ non-trivial lattice polytopes $Q_1,\dots, Q_L$   whose 
Minkowski sum is in $P$  will be referred to as  a {\it maximal
(Minkowski) decomposition in P}. The {\it dimension} of a maximal decomposition is the dimension of the Minkowski sum
$Q_1+\cdots+Q_L$. 
\end{definition}

\begin{example}
In the figure below, the first polygon, called $T_0$, has Minkowski length 1. For the second one, the Minkowski length is 2. Notice that this triangle has many maximal decompositions: the sum of two horizontal segments, the sum of two vertical segments, the sum of two diagonal segments, the sum of two standard 2-simplices, etc.
For the last polygon, the Minkowski length is 3, as there is a parallelogram inside that is a sum of three lattice segments.

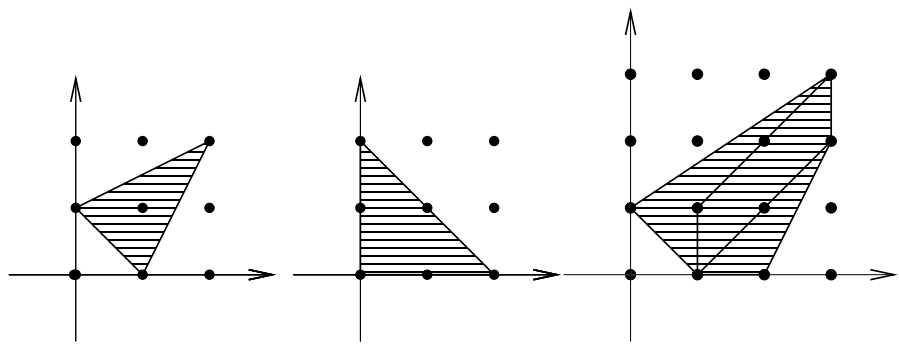
\begin{figure}[h]
\centerline{
 \scalebox{0.7}
 {
\input{T00.pstex_t}\ \ \input{2Delta.pstex_t}\input{P.pstex_t}}}
\label{F:MinkLength}
\end{figure}
\end{example}

Clearly, $L(P)$ is an $\A{n}$-invariant and the summands of every maximal decomposition in $P$ are polytopes of
Minkowski length 1.

It does not  seems  feasible to describe
polytopes of Minkowski length 1  in general. However, in
dimension 2   such a description is given in \cite{SS1} and we  
reproduce it here. 


 \begin{figure}[h]
\centerline{
 \scalebox{0.55}
{
\input{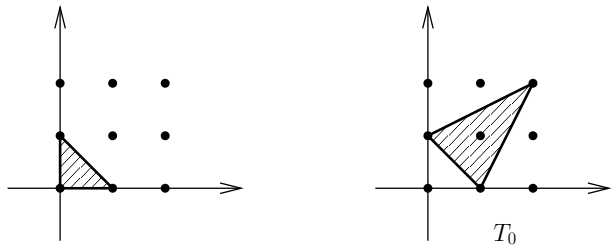}}}
\caption{Polygons of Minkowski length 1}
\label{F:indecomp}
\end{figure}

\begin{theorem}\label{T:indecomp}\cite{SS1}
 Let $P$ be a convex lattice polygon in the plane with $L(P)=1$.  Then $P$ is $\A2$-equivalent
 to a primitive segment,  the standard 2-simplex $\mathcal{D}$ or the triangle $T_0$ with
vertices $(1,0)$, $(0,1)$ and $(2,2)$. 
\end{theorem}

It is also proved in \cite{SS1}  that a maximal decomposition $Q\subseteq P$ can have at most one summand $\A2$-equivalent to $T_0$, and if this is the case, the
remaining summands are  $[0,e_1]$, $[0,e_2]$, and $[0,e_3]$, that is, $Q$ is $\A{2}$-equivalent to $Q=T_0+n_1[0,e_1]+n_2[0,e_2]+n_3[0,e_1+e_2]$. Here $e_1,e_2$ are the standard basis vectors.

We will recover this result (using a new method that will be later  applied to the 3D case) and will also show that if a triangle $\Delta$  is a summand of a maximal decomposition of a polygon $P$, then the other summands are either primitive segments, or exactly that triangle $\Delta$. That is, if $Q_1$ and $Q_2$ are triangles and $L(Q_1+Q_2)=2$, then $Q_1=Q_2$. This refinement will be important for our  3D discussion.

Before we state the result,  we set notation and prove a lemma which will also be important for our future discussion in dimension 3.
Let $P$ be a lattice polytope in $\R^n$.  For each segment whose endpoints are lattice points in $P$,  consider its direction vector reduced modulo 3.  Since $v$ and $-v$ define the same segment, we identify such vectors. Using this equivalence relation, we obtain  the set $\Z_3 \P^{n-1}$ of equivalence classes.

\begin{lemma}\label{L:mult} Let $L(P)=L(Q)=1$ and $L(P+Q)=2$, where $P$ and $Q$ are lattice polytopes in $\R^n$. Then if $P$ and $Q$ each have  a segment of some class $a$, then those two segments are equal (are the same  up to translation). If $P$ has at least two segments of class $a$, then $Q$ has no segments from that class.  
\end{lemma}

\begin{proof}
If  $P$ and $Q$ both have lattice segments from the same equivalence class, then, unless these segments are equal, their Minkowski sum  contains a segment of Minkowski length 3 (since either sum or difference of the direction vectors is a multiple of 3). If $P$  has multiple segments from one class, then these segments cannot be translates of each other, as they would form a parallelogram in $P$ of Minkowski length at least 2. If $Q$ has a segment from that class, it would be not a translate  of at least one of the two segments in $P$ and we again conclude $L(P+Q)\geq 3$.
\end{proof}

\begin{theorem} Let $P$ be a convex lattice polygon. If one of the summands  of a maximal decomposition $Q$ in $P$ is $\A{2}$-equivalent to $T_0$, then $Q$ is $\A{2}$-equivalent to $Q=T_0+n_1[0,e_1]+n_2[0,e_2]+n_3[0,e_1+e_2]$. If one of the summands $\Delta$ of $Q$ is $\A{2}$-equivalent to the standard 2-simplex , then the remaining summands that are not primitive segments, are equal to  $\Delta$.
\end{theorem}

\begin{proof}
We have four equivalence classes in  $\Z_3 \P^1$:
$$(1,1),  (1,-1), (1,0), (0,1).
$$ 
Notice that if $a$ and $b\in\Z_3\P^1$ are linearly independent (that is, $a\neq \pm b$), then they generate all the classes:
$$\langle a,b \rangle= \{ a, b,  a+b, a-b\}.$$
Now let one of the summands in a maximal decomposition in $P$ be $\A{2}$-equivalent to $T_0$. Then  we can assume this summand is exactly $T_0$. The direction vectors of  the lattice segments in $T_0$ are  $(1,0)$, $(0,1)$, $(1,1)$, $(1,-1)$, $(1, 2)$, and $(2,1)$. The last three are all from the same class, so by Lemma~\ref{L:mult} segments from this class cannot show up in other summands. The first three are all from distinct classes, hence by the lemma only  segments with direction vectors $(1,0)$, $(0,1)$, and $(1,1)$ can show up in other summands. Since all four classes are covered, we have shown that the direction vectors of lattice segments in other summands can only be   $(1,0)$, $(0,1)$, and $(1,1)$. One can use such segments to form a triangle in four different ways. The result will be the standard 2-simplex and its reflections. In each of these four cases, it's easy to check that the Minkowski sum of such a  triangle with $T_0$ is 3, which proves our first statement.

Next, let one of the summands be equivalent to  the standard 2-simplex $\Delta$, so we can assume it's exactly $\Delta$. The direction vectors $(1,0), (0,1), $ and $(-1,1)$ are all from distinct classes, hence if other summands have lattice segments from these classes,  they would have to have these direction vectors.   If there is another triangle in the maximal decomposition, it would have to be equivalent to the standard 2-simplex, as $T_0$ is not possible  by the above argument. The direction vectors would have to belong to three distinct classes, so at least two of the sides would have to  have direction vectors  $(1,0), (0,1),$ or $(-1,1)$.  Here are eight triangles that could be formed in this way: 

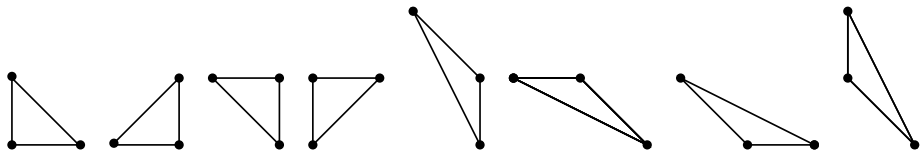
\begin{figure}[h]
\centerline{
 \scalebox{0.7}
 {
\input{eight.pstex_t}}}
\end{figure}
The last four have a segment with a direction vector either $(-2, 1)$ or $(-1, 2)$, which are from the same class with $(1,1)$, so the Minkowski sum of any of these 
triangles with $\Delta$ is at least 3. For all the remaining triangles, except $\Delta$ itself we easily check that their sum with $\Delta$ has Minkowski length 3.

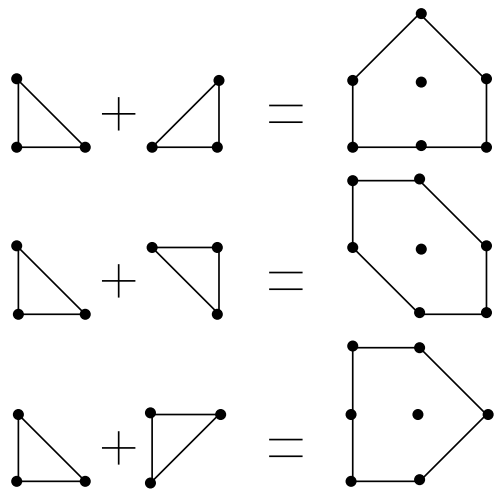
\begin{figure}[h]
\centerline{
 \scalebox{0.7}
 {
\input{sumDelta.pstex_t}}}
\end{figure}
\end{proof}

\begin{corollary}\label{C:sharetriang}\ Let $P, Q$ be  lattice polytopes in $\R^3$  with $L(P+Q)=2$. Consider the intersection of a plane $\pi$  with each $P$ and $Q$.
If each $\pi\cap P$ and $\pi\cap Q$ contains  a lattice triangle, then 
these lattice triangles are the same up to translation.
\end{corollary}
\begin{proof}
Let $u$ be a primitive normal vector to  $\pi$. Let $A\in {\rm GL}(3,\Z)$ be a matrix whose last row is $u$.  (It is shown, for example, in \cite{Newman}, Theorem II.1,  why such a matrix exists.)
Then $A$ maps $\pi$ to the $(x,y)$-plane and the result follows from the previous theorem.
\end{proof}

%
%
\section{Algorithm for Computing $L(P)$ for 3D polytopes.}

It  was shown in \cite{SS1} that in the plane case there always exists a maximal decomposition in $P$ of a very simple form. Namely, there exists a maximal decomposition that is  equivalent to $n_1[0,e_1]+n_2[0,e_2]+n_3[0,e_1+e_2]$ for some $n_1, n_2, n_3\in \N$.  This fact was used in \cite{SS1} to build an algorithm for finding $L(P)$.   To extend this result to the 3D case, we  first make a definition.

\begin{definition}
Let $P\subset\R^n$ be a convex lattice polytope. Then the set of its maximal decompositions is partially ordered by inclusion. That is, we say that 
$$A+P_1+\cdots+P_k< B+Q_1+\cdots+Q_l$$ 
if  $A+P_1+\cdots+P_k\subsetneq B+Q_1+\cdots+Q_l$. Here $A$ and $B$ are points in $\Z^n$. A maximal decomposition is called a {\it smallest maximal decomposition} if it is minimal with respect to
this partial order. Note that a smallest maximal decomposition is a Minkowski sum of segments.

\end{definition}

\begin{proposition}\label{T:2d1}
Let $P\subseteq \R^2$ be a lattice polygon. Consider a smallest maximal  decomposition $Z$ in $P$: 
$$P\supseteq Z = A+ n_1 E_1+\cdots+n_l E_l.
$$ 
 Then ${\rm Area}(E_i+E_j)\leq 1$ for any choice of $1\leq i,j\leq l$.
 \end{proposition}
\begin{proof}
Let $v_1$ and $v_2$ be the primitive direction vectors of the segments  $E_i$ and $E_j$ and assume that the area of the parallelogram spanned by $v_1$ and $v_2$ is at least 2. Applying an $\A{2}$ transformation, we can assume that $A$ is the origin, $v_1=e_1=(1,0)$ and $v_2=(a,b)$, where $0\leq a <b$ and $b>1$, which implies that $(1,1)\in\Pi=[0,e_1]+[0,(a,b)]$.   We show now that there is always a segment $I$ of Minkowski  length 2 that lies strictly inside $\Pi$ and hence, we can pass from $\Pi$ to  $2I$ and get a smaller maximal decomposition.
 If both $a$ and $b$ are even then 
$2[0,(a/2,b/2)]$ is strictly inside of $\Pi$;  if $a$  is odd and $b$ is even then 
$2[0,((a+1)/2,b/2)]\subsetneq \Pi$; if $a$ is even and $b$ is odd, $(1,1)+2[0,(a/2,(b-1)/2)]\subsetneq \Pi$; if $a$ and $b$ are both odd, $(1,1)+2[0,((a-1)/2),(b-1)/2)]\subsetneq \Pi$. 
\end{proof}

\begin{theorem}\label{T:2d2}
Let $P\subseteq \R^2$ be a lattice polygon. If $Z$ is a smallest maximal decomposition in $P$, then it is $\A{2}$-equivalent to
$$P\supseteq Z = n_1 [0,e_1]+n_2[0,e_2]+n_3 [0,e_1+e_2].
$$
 \end{theorem}
\begin{proof}
Let  $P\supseteq Z = n_1 E_1+\cdots+n_l E_l$ with $v_1,\dots, v_l$  distinct primitive direction vectors of the segments $E_1,\dots, E_l$. Applying an $\A{2}$ transformation we can assume that $v_1=e_1$. Next, since $\det(v_1,v_2)=\pm 1$, we can assume that $v_2=e_2$. Then by the previous proposition, any other $v_k$ is either  $(1,1)$ or $(1,-1)$ as we can always switch a vector to its negative. Notice that  these two vectors cannot simultaneously appear in a smallest decomposition, as  the sum of the corresponding segments would contain a segment of Minkowski length 2.  Finally, these two remaining cases are $\A{2}$-equivalent.  
\end{proof}

We next treat the 3D case. 

\begin{proposition}\label{T:3d1}
Let $P\subset\R^3$ be a  lattice polytope. Consider a smallest maximal decomposition $Z$ in  $P$
$$P\supseteq Z = A+n_1 E_1+\cdots+n_l E_l.
$$ 
Then  ${\rm Vol}(E_i+E_j+E_k)\leq 2$ for any choice of $1\leq i,j,k\leq l$.
\end{proposition}

\begin{proof}
Let $v_1$, $v_2$, and $v_3$ be the primitive vectors that go along the segments  $E_i$, $E_j$, $E_k$, and assume that the volume of the parallelepiped spanned by $v_1$, $v_2$, and $v_3$ is at least 3. Applying an $\A{2}$ transformation (and using Proposition~\ref{T:2d1}), we can assume that $A$ is the origin, $v_1=e_1=(1,0,0)$, $v_2=e_2=(0,1,0)$,  and $v_3=(s,t,u)$, where $0\leq s \leq t <u$.  The volume of the parallelepiped spanned by $e_1,e_2,v_3$ is $|u|$. If $s=0$, then the area spanned by $e_2$ and $v_3$ is $|u|\geq 3$, which would contradict the minimality of  $Z$. We next  observe that the parallelepiped spanned by $e_1,e_2$, and $v_3$ is defined by
$$\Pi=\left\{(x,y,z) \in \R^3\mid 0 \leq z \leq u,\  \frac{s}{u} z \leq x \leq \frac{s}{u} z + 1, \  \frac{t}{u} z \leq y \leq \frac{t}{u} z + 1\right\}
$$
and consider the following three cases.
\begin{enumerate}
\item[Case 1.] $s\leq u/2$, $t\leq u/2$ \\
Using the description of $\Pi$ above, we can easily check that $(1,1,2)$ and $(s,t,u-2)$ are both in $\Pi$.  Hence a parallelogram with the vertices
$(1,1,0), (1,1,2), (s,t, u-2)$, and  $(s, t, u)$ is inside $\Pi$. This  parallelogram is a Minkowski sum of three segments
$$(1,1,0)+2[0,(0,0,1)]+[0,(s-1,t-1,u-2)]\subsetneq \Pi ,$$
 which contradicts the minimality of $Z$. Notice that $u\geq 3$ ensures that the segments involved in the decomposition are non-trivial.
\item[Case 2.] $s> u/2$, $t>u/2$ \\
Then $(2,2,2)$ and $(s-1,t-1,u-2)$ are in $\Pi$, so a parallelogram with the vertices  $(0,0,0), (2,2,2), (s-1,t-1,u-2)$, and  $(s+1, t+1, u)$ is inside $\Pi$.
This parallelogram   is a Minkowski sum of three segments
$$2[0, (1,1,1)]+[0, (s-1,t-1,u-2)]\subsetneq\Pi,$$ which contradicts the minimality of $Z$.
\item[Case 3.] $s\leq u/2$, $t>u/2$ \\
 Then $(1,2,2)$ and $(s,t-1,u-2)$ are in $\Pi$,  so a parallelogram with the vertices  $(1,0,0), (1,2,2), (s,t-1,u-2)$, and  $(s, t+1, u)$ is inside $\Pi$.
 This parallelogram   is a Minkowski sum of three segments
 $$(1,0,0)+2[0, (0,1,1)]+[0, (s-1,t-1,u-2)]\subsetneq\Pi,$$ 
 and we get the same contradiction again.
\end{enumerate}
\end{proof}

\begin{theorem}\label{T:3d2}
Let $P\subset \R^3$ be a lattice polytope. Let  $Z$  be a smallest maximal decomposition in $P$, then it is $\A{2}$-equivalent  to either
$$n_1[0, e_1] + n_2[0, e_2] + n_3[0, e_1+e_2+2e_3] + n_4[0, e_1+ e_2+e_3] + n_5[0, e_1 + e_3]
+n_6[0, e_2 + e_3] + n_7[0, e_3]
$$
or
$$n_1 [0,e_1] + n_2 [0,e_2] + n_3 [0,e_3] + n_4 [0,e_1 + e_2 + e_3]  + n_5 [0,e_1 \pm e_2]  + n_6 [0,e_1+e_3]  + n_7 [0,e_2 + e_3].
$$
\end{theorem}
\begin{proof}  Assume first that there are three segments in the maximal decomposition $Z$ whose direction vectors generate a parallelepiped of volume  2.
We can then assume that the first direction  vector is $e_1$. By Proposition \ref{T:2d1} we can assume that the second vector is $e_2$. Next, we can assume that the third direction vector $v$ is of the form $(s,t,u)$ where $0 \leq s \leq t <u$.  Since $2=|\det(e_1,e_2,v)|$, we know that $u=2$ and the only options for the third vector 
are $(0,1,2)$ and  $(1,1,2)$. The first of these two options is impossible, as the sum of $(0,1,2)$ and $(0,1,0)$ is not primitive,  which contradicts the minimality of $Z$. We have shown that the third vector is $(1, 1, 2)=e_1+e_2+2e_3$. 

Let  $v=(a,b,c)$ be a direction vector of some other segment in the maximal decomposition $Z$.  We know that $|\det(e_1, e_2, v)|\leq 2$, $|\det(e_1, e_1+e_2+2e_3, v)|\leq 2$, and $|\det(e_2,e_1+e_2+2e_3,  v)|\leq 2$, which gives us the following restraints on the components of $v$:  $|c - 2 b| \leq 2$,  $|c - 2 a| \leq 2$, and $|c|\leq 2$. By flipping  the direction vector $v$ if necessary, we can assume that $c\geq 0$.

If $c=0$, then $v=(1,1,0)$ or $(1,-1,0)$. Both options are impossible as then the sum of $v$ with $(1,1,2)$ is $(2,2,2)$ or $(2,2,0)$, so we can pass to a  smaller maximal  decomposition.  If $c=1$, then $v=(0,0,1)$, $(0,1,1)$, $(1,0,1)$, $(1,1,1)$.
If $c=2$, then $v=(0,1,2)$, $(1,0,2)$, $(1,2,2)$, or $(2,1,2)$. Adding either $e_1$ or $e_2$ to each of these four vectors we can get a non-primitive vector,
so none of these vectors occur in our maximal decomposition. We have shown that in the case when there are three segments in the maximal decomposition $Z$  that generate a parallelepiped of volume  2, then $Z$ is $\A{2}$-equivalent to 
$$n_1[0, e_1] + n_2[0, e_2] + n_3[0, e_1+e_2+2e_3] + n_4[0, e_1+ e_2+e_3] + n_5[0, e_1 + e_3]
+n_6[0, e_2 + e_3] + n_7[0, e_3].
$$

Next, assume that any three segments in the maximal decomposition $Z$ generate a parallelepiped of volume at most 1. 
If for any three vectors the volume is zero, then we are in the plane case and we are done. Otherwise, we can assume that 
first three vectors are $e_1$, $e_2$, and $e_3$. Let $v=(a,b,c)$ be any other direction vector in the maximal decomposition $Z$.
Then we have $|a|\leq 1$, $|b|\leq 1$ and $|c|\leq 1$. By flipping the direction vector we can assume that $c\geq 0$.
Here are the options for $v$ that we get, written in four lines: 
$$(1,1, 0), (1,-1, 0),$$
$$(0,1,1), (0,-1,1),$$
$$(1, 0, 1), (-1,0,1),$$
$$(1, 1,1), (1,-1,1), (-1,1,1), (-1,-1,1).$$
Notice that no two vectors from the same line here can occur in $Z$ together as their sum is not primitive, which would contradict the minimality of $Z$.
By flipping the direction of basis vectors, we can assume that if any of the four vectors in the last line occur in $Z$, then it is $(1,1,1)$.  Let's assume that this is the case and $(1,1,1)$ occurs in $Z$.

We notice next $(-1,0,1)$ and $(0,-1, 1)$ can not occur in $Z$ together as if we add these two vectors together with $(1,1,1)$, we get $(0,0,3)$.
We can make the same observation about $(-1,0,1)$ and $(1,-1,0)$ and then about $(0,-1,1)$ and $(1,-1,0)$. This implies that only one of $(1,-1,0)$, $(0,-1,1)$,
and $(-1,0,1)$ occurs in $Z$. By permuting $e_1,e_2$, and $e_3$ we can assume that the one that occurs is $(1,-1,0)$.

In the case when none of of the four vectors $(1, 1,1), (1,-1,1), (-1,1,1), (-1,-1,1)$ occur in $Z$, by applying a diagonal change of basis  with $\pm 1$'s  on the main diagonal (which will not change $[0,e_1], [0,e_2]$, $[0,e_3]$), we can  turn any pair of vectors from the set  $(1,-1,0), (0,-1,1), (-1,0,1)$ into corresponding vectors 
with positive entries. For example, a matrix with the diagonal entries $-1,-1,1$ will turn $(-1,0,1)$ and $(0,-1,1)$ into $(1,0,1)$ and $(0,1,1)$.  Hence we will be able to get  rid of all the vectors with negative entires except, possibly, one. By permuting $e_1$, $e_2$, and $e_3$, we can assume that the vector with a negative entry is $(1,-1,0)$. We have shown that if any three segments in  $Z$ generate a parallelepiped of volume at most 1, then $Z$ is $\A{3}$-equivalent to either 
$$n_1 [0,e_1] + n_2 [0,e_2] + n_3 [0,e_3] + n_4 [0,e_1 + e_2 + e_3]  + n_5 [0,e_1+ e_2]  + n_6 [0,e_1+e_3]  + n_7 [0,e_2 + e_3]
$$
or
$$n_1 [0,e_1] + n_2 [0,e_2] + n_3 [0,e_3] + n_4 [0,e_1 + e_2 + e_3]  + n_5 [0,e_1-e_2]  + n_6 [0,e_1+e_3]  + n_7 [0,e_2 + e_3].
$$
\end{proof}

Notice that in 2D a smallest maximal decomposition has at most 3 distinct summands; in 3D, as we have just shown, such a decomposition has at most 7 distinct summands. It turns out that in dimension $n$ a smallest maximal decomposition  has at most $2^{n}-1$ distinct summands.

\begin{proposition}\label{P:bound} Let $P\subset\R^n$ be a convex lattice  polytope. Let $Z$ be a smallest maximal decomposition in $Z$. 
Then $Z$ has at most $2^n-1$ distinct summands.
\end{proposition}
\begin{proof} Reduce all the summands in $Z$ modulo 2. Since the summands are primitive segments, there will be $2^n-1$ possibilities for a reduced segment.
If the number of distinct  segments in $Z$ is at least $2^n$, we will have two summands that are equal modulo 2. Then their sum is non-primitive, which contradicts the minimality of  $Z$.
\end{proof}

Although we expect that  the sum of the $2^n-1$ segments  with $0,1$ components mentioned in the proof of the above proposition has Minkowski length $2^n-1$, we do not have a proof of this statement, which would have implied that the bound of the proposition is sharp.

Let a lattice polytope $P$ be described by its  facets equations. Then Barvinok's algorithm \cite{Bar, Koppe}  counts the number of lattice points in $P$ in polynomial time.
We will assume that the list $P\cap\Z^3$ of the lattice points in $P\subset\R^3$ is given, and will explain how to find the Minkowski length of $P$ in polynomial time in $P\cap\Z^3$.

\begin{theorem}\label{T:alg}
 Let $P\subset \R^3$ be a lattice polytope with the given set of its lattice points $P\cap\Z^3$.  Then the Minkowski length $L(P)$ can be
found in polynomial time in $P\cap\Z^3$.
\end{theorem}
\begin{proof}

This algorithm relies on Theorem~\ref{T:3d2}. We search for all possible decompositions of the form described in the theorem.
For every quadruple of points $\{A,B,C,D\}\subseteq P\cap\Z^3$, where it is
important which point goes first and the order of the other three does
not matter, we check if $[0,B-A]$, $[0,C-A]$, and $[0,D-A]$ generate a
parallelepiped of volume one or two. If the volume is one,  these segments are equivalent
to $[0,e_1]$, $[0,e_2]$, $[0,e_3]$ and we look for maximal decompositions equivalent to
$$n_1 [0,e_1] + n_2 [0,e_2] + n_3 [0,e_3] + n_4 [0,e_1 + e_2 + e_3]  + n_5 [0,e_1 \pm e_2]  + n_6 [0,e_1+e_3]  + n_7 [0,e_2 + e_3],
$$
that is, maximal decompositions of the form
$$n_1E_1+ n_2 E_2+ n_3 E_3 + n_4 E_4+ n_5 E_5  + n_6 E_6  + n_7 E_7,
$$
where $E_1=[0,B-A]$, $E_2=[0,C-A]$, $E_3=[0,D-A]$, $E_4=[0,B+C+D-3A]$, $E_5=[0,B+C-2A]$ or $[0, B-C]$, $E_6=[0,B+D-2A]$, 
$E_7=[0, C+D-2A]$.

If the volume is two, we check if the segments $[0,B-A]$, $[0,C-A]$, and $[0,D-A]$ are primitive and if any  two of them generate
a parallelogram whose only lattice points are the vertices.  If this is the case, these three segments are equivalent to $[0,e_1]$, $[0,e_2]$, 
and $[0, e_1+e_2+2e_3]$. We then let $E_1=[0,B-A]$, $E_2=[0,C-A]$, $E_3=[0,D-A]$, $E_4=[0,(B+C+D-3A)/2]$, $E_5=[0,(B+D-C-3A)/2]$,
$E_6=[0,(C+D-B-3A)/2]$, $E_7=[0,(D-A-B-C)/2]$.

Next, for every $1\leq i\leq 7$, we find $M_i$, the largest integer such that  there is some lattice point $F$ in $P$ with $F+M_iE_i\subseteq P$.
For each $7$-tuple of integers $m=(n_1, \dots, n_7)$ where $0\leq
n_i\leq M_i$,
we check if some lattice translate of the zonotope
$Z_m=n_1E_1+\cdots+n_7E_7$ is contained in $P$ (we run through lattice
points $F$ in $P$ to check if  $F+Z_m$ is contained in $P$). For all
such zonotopes that fit into $P$ we look at $n_1+\cdots+n_7$ and find the
maximal possible value $N$ of this sum.

Finally, the largest such sum $N$ over all choices of
$\{A,B,C, D\}\subseteq P\cap\Z^3$ is $L(P)$. Clearly, this algorithm is polynomial in $P\cap\Z^3$.
\end{proof}

A group of REU students (Ian Barnett, Benjamin Fulan,  and Candice Quinn) at Kent State University in Summer 2011 tried to generalize this algorithm to dimension 4.
Their first step was to obtain a 4D version of Proposition~\ref{T:3d1}.  They showed that if $P\subseteq\R^4$ is a  lattice polytope
and $Z= A+n_1 E_1+\cdots+n_l E_l$ is a smallest maximal decomposition in $P$, then 
 ${\rm Vol}(E_i+E_j+E_k+E_m)\leq 14$ for any choice of $1\leq i,j,k,m\leq l$, and this bound is sharp.
Unfortunately, this bound is too high to obtain a description of smallest maximal decompositions in 4D, similar to the one  of 
 Proposition~\ref{T:3d2}.

\section{Lattice Polytopes of Minkowski Length 1}

It was shown in Theorem 1.6 of  \cite{SS1}  that if $P\supseteq Q=Q_1+\cdots+Q_l$ is a maximal decomposition of a polygon $P$ then at most one of $Q_i$ has an integer lattice point in its interior, that is, $\sum I(Q_i)\leq 1$. This fact was crucial in~\cite{SS1} for establishing  bounds on the minimum distance of the  toric surface code defined by $P$. We expect  that in order to extend these bounds  to 3D codes, one needs to explore similar questions in dimension 3. As it was mentioned above, a description of  polytopes of Minkowski length 1 in dimension 3 does not seem feasible. We will instead  reduce  the lattice segments contained in a 3D lattice polytope modulo 3, which will help us show that if $P$ is a 3D polytope with $L(P)=1$, then $\sum I(Q_i)$ is at most 4.


Let $P$ be  a lattice polytope in $\R^3$ of Minkowski length 1. Then $P$  has at most 8 lattice points. Indeed, otherwise there would have been two lattice points in $P$ that are congruent modulo 2, and hence the segment connecting them would have  Minkowski length of at least 2. 

For each segment whose endpoints are lattice points in $P$,  we consider its direction vector reduced modulo 3.  Since $v$ and $-v$ define the same segment, we identify such vectors. For thus defined  modulo 3 segments there are 13  equivalence classes:
$$(1,1,1), (1,1,-1), (1,1,0), (1,-1,1),(1,-1,-1), (1,-1,0),$$
$$ (1,0,1),(1,0,-1),(1,0,0),(0,1,1), (0,1,-1), (0,1,0), (0,0,1), 
$$ 
that is, we are dealing with  the projective space $\Z_3 \P^2$. Notice that if $a, b,$ and $c\in\Z_3\P^2$ are linearly independent (that is,  $a\neq b$, and  $c\notin \langle a,b\rangle=\{ a, b, a\pm b\}$),  they generate all the classes:
$$a, b, c, a+b, a-b, a+c, a-c, b+c, b-c, a+b+c, a+b-c, a-b+c, -a+b+c.
$$

%
%

Let $S$ be a five-point lattice set contained in a polytope of Minkowski length 1. There are ten lattice segments that connect lattice points in $S$.  We will classify such sets $S$ according to the numbers of segments from distinct classes in $\Z_3 \P^2$.

\begin{proposition}\label{T:class} If $L(P)=1$, then any 5-point lattice set $S$ in $P$ is of one of the following types.
\begin{itemize}
\item ${\rm 4+2+2+2}$  The segments are from classes $4a, 2b, 2(a+b), 2(a-b)$.  Here a 4 or a 2 in front of segment's class denotes its multiplicity, which is the number of times it  occurs  among the lattice  segments in $S$.
\item ${\rm 3+3+2+2}$  The segments are from classes $3a,3b, 2(a+b), 2(a-b)$.
\item ${\rm 3+(7)}$ The segments are from classes $3a, b, a+b, a-b, c, a+c, a-c, a+b-c$.
\item ${\rm 2+2+(6)}$ The segments are from classes $2a, 2b, a+b,a-b, a+c, b+c, a+b+c, c$. 
\item ${\rm(10)}$ All ten lattice segments connecting points in $S$ are from distinct classes in  $\Z_3 \P^2$.
\end{itemize}
The elements $a,b,c\in\Z_3 \P^2$ in each of the type descriptions are linearly independent. All types except for the last one have segments from classes $a, b, a+b, a-b$.
\end{proposition}

\begin{proof}
We assign direction to the segments by picking a standard representative from each of the classes.  If two segments from the same class share a vertex, the arrows cannot both point to or away from the vertex as in this case the third side in the triangle is of Minkowski length at least 3.  We also notice that if two sides in a triangle are from the same class, then the third one is also from that class and we get the  triangle diagram below.

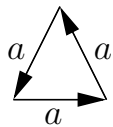
\begin{figure}[h]
\centerline{
 \scalebox{.65}
 {
\input{three.pstex_t}}}
\end{figure}

{\noindent}No other segment starting in one of these three vertices can be of class $a$, so the only remaining segment in $S$ that could be of class $a$, is the one connecting two remaining points of $S$. Hence the largest number of segments of the same class in $S$ is 4. If we have 4 segments of the same class we get the diagram below.


\begin{figure}[h]
\centerline{
 \scalebox{.6}
 {
\input{twoplustwo2.pstex_t}}}
\caption{4+2+2+2}
\end{figure}

{\noindent} We call this type 4+2+2+2 as there are 4 segments of one type and 2 segments of each of the three other types.

Next, assume we only have 3 segments from class $a$. Then they would have to form a triangle. We could also have another 3 segments of class $b$, forming a triangle sharing a vertex with the first triangle. Then there are 3 segments of class $a$, 3 of class $b$, and 2 of each of $a+b$ and $a-b$. We call this type 3+3+2+2.

 Assume next there is no other triangle. Connect one of the vertices of the triangle whose sides are of class $a$ to a fourth lattice point  in $S$. Let this segment be of class $b$. The segment connecting the fourth lattice point to the fifth cannot be from classes $a,b, a+b,$ or $a-b$, as this would give either another triangle or  four segments of the same class. Hence that segment is of class $c$, such  that the set $\{a,b,c\}$ is linearly independent and we get the diagram below.
\begin{figure}[h]
\centerline{
 \scalebox{.6}
 {
\input{threeplusseven.pstex_t}}}
\caption{3+(7)}
\end{figure}
We call this type 3+(7). 

If there are no 3s but there is a 2, we get a configuration of type $2+2+(6)$. 
 
 \begin{figure}[h]
\centerline{
 \scalebox{.6}
 {
\input{twoplustwo1.pstex_t}}}
\caption{2+2+(6)}
\end{figure}
 
 Finally, it is possible that there are no repeats among classes of segments. An example of this situation is a tetrahedron with the vertices $(1,0,0), (0,1,0), (0,0,1), (-1,-1,-1)$ with one lattice point, the origin, strictly inside. We call this type (10).
\end{proof}

\begin{lemma}\label{L:int} Let $a,b,c,d\in \Z_3 \P^2$ where $a\neq b$ and $c\neq d$. Then $\langle a,b\rangle\cap \langle c,d\rangle\neq\emptyset$.
\end{lemma}
\begin{proof}
If $c\in\langle a,b\rangle$ the conclusion is obvious, so we can assume that $\langle a,b,c\rangle= \Z_3 \P^2$.  One of $d,c+d, c-d$ does not have $c$ in its expression in terms of $a,b,c$, hence it belongs to $\langle a,b\rangle$.
\end{proof}

\begin{proposition} Let $P$ and $Q$ be 3D lattice polytopes of Minkowski length 1 with at least five lattice points each. If there exists a 5-point  lattice subset $S$   of $P$ of type {\rm 4+2+2+2} or {\rm 3+3+2+2}, then  $L(P+Q)\geq 3$.
\end{proposition}
\begin{proof}
Pick any 5-point lattice subset $T$ of $Q$. Since in $S$ we have used up four classes with multiplicities greater than 1, by Lemma~\ref{L:mult}, $T$ cannot be of type $(10)$, since the overall number of classes is 13. In $S$, we have multiple segments of each of the  classes $a$, $b$, $a+b$, $a-b$ for some $a,b\in\Z_3\P^2$. Since  $T$ is not of type (10),  we also have segments of classes $c$, $d$, $c+d$, $c-d$ in $T$ for some $c,d\in\Z_3\P^2$, not necessarily with multiplicities. By Lemmas \ref{L:int} and \ref{L:mult} we conclude $L(P+Q)\geq 3$.
\end{proof}

\begin{proposition}
Let $P$ and $Q$ be 3D lattice polytopes of Minkowski length 1. If there exist 3-point  lattice subsets $S$ and $T$  of $P$ and $Q$ correspondingly, each of which forms a triangle with sides of the same class, then  $L(P+Q)\geq 3$.  In particular, if both $S$ and $T$ are of type $3+(7)$, then $L(P+Q)\geq 3$.
\end{proposition}

\begin{proof}
We first notice that if a lattice polytope contains a lattice triangle with sides of the same class, then this triangle  is equivalent to $T_0$.
Indeed, we can easily map this triangle to one in the $(x,y)$-plane by creating a matrix of determinant 1 whose last row is a primitive vector 
orthogonal to the plane of the triangle. We know that in the  $(x,y)$-plane, up to the equivalence,  there are only two triangles of length one, the unit triangle and $T_0$. The unit triangle has sides that belong to three distinct classes and the sides of $T_0$ are all from the same class.
Hence the initial triangle is equivalent to $T_0$ and, therefore, has a lattice point inside and all four points are in the same plane.
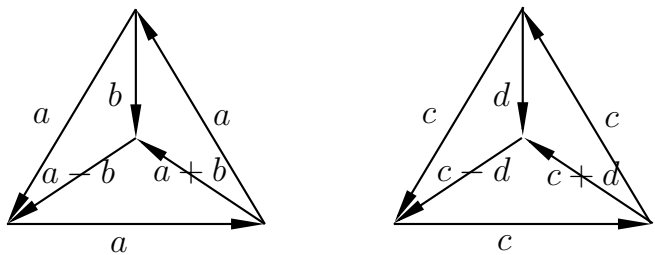
\begin{figure}[h]
\centerline{
 \scalebox{.9}
 {\input{threeplusthree.pstex_t}}}
\end{figure}

We have such a configuration in both $P$ and $Q$. Let the sides of the triangle in $Q$ be of class $c$
 and one of the segments inside this triangle be of class $d$. By Lemma~\ref{L:int}, $P$ and $Q$ share a segment. By Lemma~\ref{L:mult},
 they cannot share $a$ or $c$, so they have a common lattice segment inside the triangles. We can assume  that $b$ and $d$ represent parallel segments.
 Notice that when extended to the intersection with the opposite side of $T_0$, these segments   have Minkowski length 1.5, that is, if we add them up 
 we get a segment of Minkowski length 3.  Hence $L(P+Q)\geq 3$. For example, if $P=Q=T_0$  we get the diagram below.
 
 
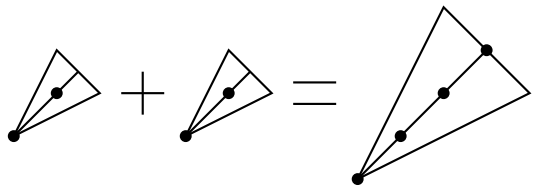
\begin{figure}[h]
\centerline{
 \scalebox{.9}
 {\input{2T0.pstex_t}}}
\end{figure}
 \end{proof}

\begin{proposition} Let $P$ and  $Q$ be 3D lattice polytopes of Minkowski length 1. If there exist 5-point  lattice subsets $S$ and $T$ of $P$ and $Q$
of types $2+2+(6)$ and  $3+(7)$ correspondingly,  then $L(P+Q)\geq 3$.
\end{proposition}
\begin{proof}
We assume that $L(P+Q)=2$. 
Let the multiple classes in $S$ be  $a$ and $b$ and the class of multiplicity 3 in $T$ be $d$.  
Let the triangle in $T$ with all sides of class $d$ be $ABC$ and the remaining  two lattice points in $T$ be $D$ and $E$ with $AD$ of class $e$ and $BE$ of class $f$, as depicted in the diagram below.

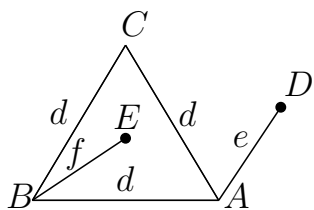
\begin{figure}[h]
\centerline{
 \scalebox{.65}
 {\input{prop35.pstex_t}}}
\end{figure}

By  Lemma~\ref{L:int},   $\langle a, b\rangle \cap \langle d, e\rangle\neq\emptyset$ and $\langle a, b\rangle \cap \langle d, f\rangle\neq\emptyset$.
Since classes $a$ and $b$ cannot occur in $T$ and class $d$ cannot occur in $S$, segments of classes $a+b$ and $a-b$ have to appear in $T$, and we can assume that $e=a+b$ and $f=a-b$. This is because  these two segments in $T$ cannot share a vertex, as the third side in the triangle formed by  $a+b$ and $a-b$ would have to be of class  either $a$ or $b$.  

By Proposition~\ref{T:class}, the segments connecting the lattice points in $S$ are of classes $a$, $b$, $a+b$, $a-b$, $a+c$, $b+c$, $a+b+c$, and $c$ for some linearly independent $a,b,c\in\Z_3 \P^2$.
Hence there are five options left for $d$: $a-c$, $b-c$, $a+b-c$, $a-b-c$, $a-b+c$.  Notice that  in $T$ we have segments of classes $d+a+b$, $d-a-b$, $d-a+b$, and $d+a-b$.
Going through the five options for $d$, we observe that every time there are four lattice segments that are shared between $S$ and $T$. Two of them are $a+b$ and $a-b$. The remaining two for each of the five cases are listed in the table below. 

\begin{center}
\begin{tabular}{c|c}
$d$& shared segments in $S$ and $T$\\
\hline
$a-c$&$a-c-(a+b)=-(b+c), a-c+(a-b)=- (a+b+c)$\\
$b-c$&$b-c-(a+b)=-(a+c), b-c-(a-b)=-(a+b+c)$\\
$a+b-c$&$a+b-c-(a-b)=-(b+c),  a+b-c+(a+b)=-(a+b+c)$\\
$a-b-c$&$a-b-c-(a-b)=-c$, $a-b-c+(a+b)=-(a+c)$\\
$a-b+c$&$a-b+c-(a-b)=c$, $a-b+c-(a+b)=b+c$
\end{tabular}
\end{center}

We have checked that there are always at least four segments shared between $S$ and $T$, with the extra condition that in $T$ none of these four segments is  $ED$.
In $T$,  three of these segments cannot all  have  $E$ as an endpoint, as this would imply that  two of these segments in $S$ also share an endpoint, so by Corollary~\ref{C:sharetriang} there is a shared triangle, one of whose sides is $d$, which is impossible. Similarly, three of the shared  segments cannot  all have  $D$ as an endpoint.  Hence there are two possible scenarios, depicted in the diagram below. The shared segments are marked by $a+b,a-b,x$, and $y$. 

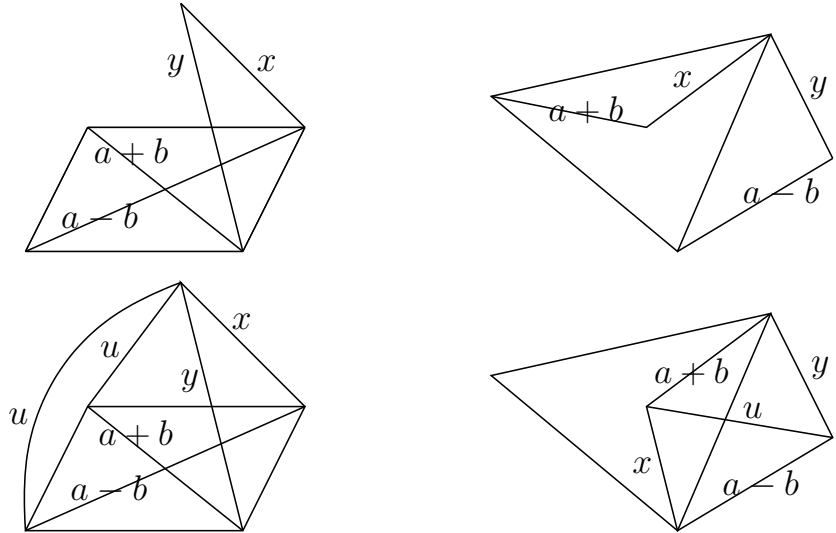
\begin{figure}[h]
\centerline{
 \scalebox{.65}
 {\input{prop352.pstex_t}}}
\end{figure}

In the first scenario, the triangle formed by $x$ and $y$ in $S$ is shared, so its third side $b$ is also shared, which leads to a contradiction.
In the second scenario, two triangles, one formed by $a+b$ and $y$, and  another by $a-b$ and $x$ are shared, so their third side $u$ appears twice in $S$, which is impossible.
\end{proof}

\begin{proposition}
Let $P$ and $Q$ be 3D lattice polytopes of Minkowski length 1. If there exist 5-point  lattice subsets $S$ and $T$  of $P$ and $Q$ correspondingly, of type {\rm 2+2+(6)} each, then  $L(P+Q)\geq 3$.
\end{proposition}

\begin{proof}

We assume that $L(P+Q)=2$. Let the $a$ and $b$ be the classes of lattice segments in $S$  of multiplicity 2. Then $S$ also has segments of classes $a+b$ and $a-b$. By Lemma~\ref{L:mult} segments in $T$ of multiplicity 2 cannot be of classes $a,b,a+b,a-b$, so we  can assume that one of them is $c$ and the other  is $a+b+c$, by switching direction vectors of $a$, $b$, and $c$, if needed. Since we have 13 classes total, $S$ and $T$ overlap in at least 3 segments. One of them is of class $a+b$.  We will show that $S$ and $T$ share three segments that form a triangle one of whose sides is of class either $a+b$ or $a-b$.

Assume that the segment of class  $a-b$ is also shared. A segment of class $a-b$ in $T$ cannot share a vertex with a segment of class $a+b$ (then either the sum or difference of those classes would also be represented in $S$, but classes $a$ and $b$ cannot appear in $S$).
Changing direction of $c$ and/or  both $a$ and $b$, we can assume that $a-b$ is as in the diagram below.

 
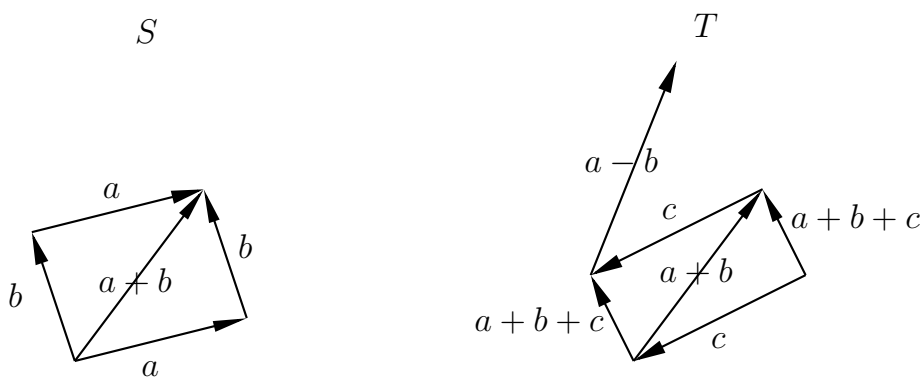
\begin{figure}[h]
\centerline{
\scalebox{.9}
 {
\input{shared.pstex_t}}}
\end{figure}

Let the class of the third shared segment be $x$. Assume that in $S$ the segment of class $x$ shares a vertex with $a-b$. Then in $T$ the segment of class  $x$ would  have to share
a vertex with $a-b$ as the only ones that don't are $c$, $a+b$, and $a+b+c$ and $x$ cannot be one of them. Hence $a-b$ and $x$ share a vertex in both $S$ and $T$ and therefore by Corollary~\ref{C:sharetriang} $S$ and $T$ share  a triangle with sides $a-b$, $a+b-c$, and $a+c$. Hence $x=a+b-c$ and the diagrams for $S$ and $T$ are below.

 
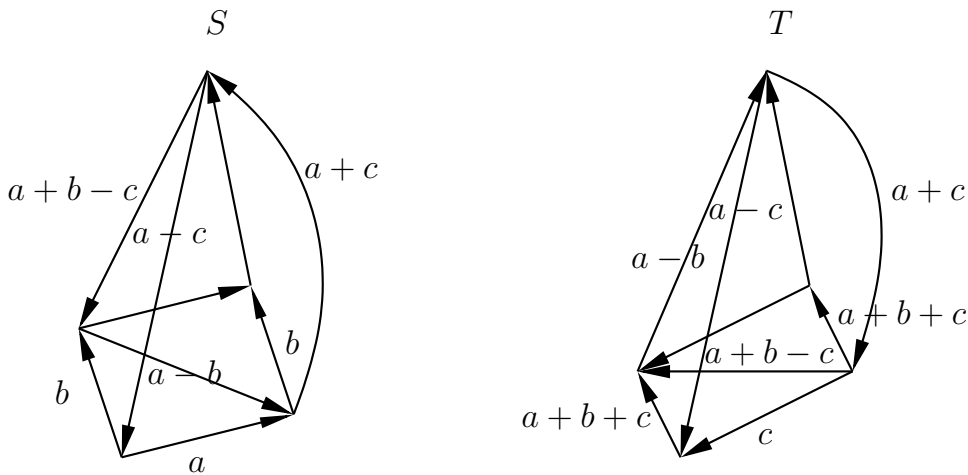
\begin{figure}[h]
\centerline{
\scalebox{.9}
 {
\input{shared2.pstex_t}}}
\end{figure}

We see that both $S$ and $T$ have triangles with $a-c$ and $a+c$ as sides, but those triangles are not identical, which contradicts $L(P+Q)=2$.

Next, let next $x$ share a vertex with $a+b$ in $S$. If $x$ does not share a vertex with $a+b$ in $T$ as well, then the options for $x$ are $a+c$ and $a+b-c$. If $x=a+c$ then either $x-a=c$ or $x+b=a+b+c$ is in $S$, which is impossible. If $x=a+b-c$, then either $x+a+b=a+b+c$ or $x-(a+b)=c$ is in $S$, which is also impossible.
Hence a triangle with base $a+b$ is shared between $S$ and $T$, which leads to the same diagram and the same contradiction as before.

It remains to consider the case when $a-b$ is not shared between $S$ and $T$. We can also assume that $a+b-c$ is not shared. (If it is shared replace $S$ and $T$ in the above argument.)  Then $S$ and $T$ share two segments both of which have the fifth point as an endpoint in both $S$ and $T$. Then the triangle formed by these two segments  is shared and the base of that triangle is $a+b$. Notice that there is no room for other shared segments  as they would have to have a fifth point
as a vertex and $a+b$ is the only option for shared base. Let's denote one of the shared segments by $d$. We get the following diagram below.

 
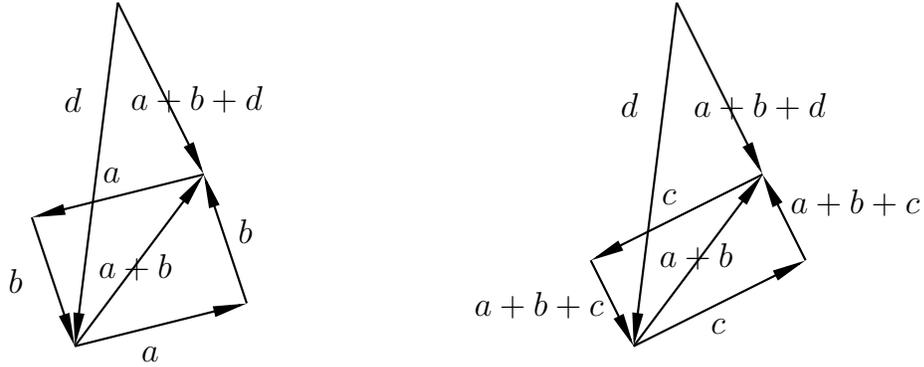
\begin{figure}[h]
\centerline{
 \scalebox{.9}
 {
\input{two226.pstex_t}}}
\end{figure}

We next search for the expression of $d$ in terms of $a,b$ and $c$ so that the only common segments between $S$ and $T$ are $a+b$, $d$, and $a+b+d$. Below is the list of segments used in $S$ and $T$.

\begin{center}
\begin{tabular}{c|c}
$S$& $T$\\
\hline
$2a$&$2c$\\
$2b$&$2(a+b+c)$\\
$a+b$&$a+b$\\
$a-b$&$a+b-c$\\
$d$&$d$\\
$d+a$&$d-c$\\
$d+b$&$d+a+b+c$\\
$d+a+b$&$d+a+b$
\end{tabular}
\end{center}

We clearly have $d\neq \pm a, \pm b, \pm (a+b), \pm (a-b), \pm c, \pm (a+b+c), \pm (a+b-c)$. All the remaining options are also very easy to get rid of. If $d=\pm b+c, \pm a+c, \pm (a-b)+c$, then $d-c\in\langle a+b \rangle$. If $d=-a-c$, then $d+a$ appears in both $S$ and $T$. If $d=-b-c$, then $d+b$ appears in both $S$ and $T$. If $d=a+b-c$, then $d+a+b$ appears in both $S$ and $T$. If $d=a-b-c$, then $d+a+b+c$ appears in both $S$ and $T$. If $d=a-c$, then $d+a$ is of the same class as $d-c$. Finally, if $d=b-c$, then $d+b$ is of the same class as $d-c$. Every time we arrive at a contradiction, which proves the proposition.
\end{proof}

\begin{proposition} Let $P$ and $Q$ be two 3D lattice polytopes of Minkowski length 1. If both $P$ and  $Q$ are of type (10) and $L(P+Q)=2$, then $P$ and $Q$ are equal  (the same up to translation). 
\end{proposition}

\begin{proof}
Since there are 13 classes of segments total, $P$ and $Q$ share at least 7 segments. Among these shared segments we can find three that have a common endpoint in $P$. At least two of these three shared segments have  a common endpoint in $Q$. Hence $P$ and $Q$ share a triangle $ABC$.  Let the two remaining lattice points in $P$ be $D$ and $E$. 

At least one of the two tetrahedra $ABCD$ and  $ABCE$ with base $ABC$, say, $ABCD$, has at least two lateral edges that are shared. Together with a segment in the base these two edges form a triangle with shared sides.  At least two of these three shared sides share  a vertex in $Q$, so by Corollary~\ref{C:sharetriang} we get a shared triangle. We have shown that $P$ and $Q$ have two pairs of shared triangles that have a common edge, so $P$ and $Q$ share a tetrahedron $ABCD$. Let the fifth vertices in $P$  and $Q$ be $E$ and $E'$.

Since $P$ and $Q$ share at least  seven lattice segments, there is a segment in $P$ connecting one of $A, B, C, D$ to $E$ which is shared with $Q$.
Let this segment be $AE$.  The parallel segment in $Q$  is of the form $E'X$ where $X=A, B, C,$ or $D$. In any of these cases, $AE$ and $E'X$ are adjacent 
to a shared segment $AX$, so triangles $AEX$ and $AE'X$ are translates of each other, which implies that $P$ and $Q$ are the same up to a translation . 
\end{proof}

\begin{proposition} Let $P$ and  $Q$ be 3D lattice polytopes of Minkowski length 1. If there exist 5-point  lattice subsets $S$ and $T$ of $P$ and $Q$
of types $3+(7)$ and  $(10)$ correspondingly,  then $L(P+Q)\geq 3$.
\end{proposition}
\begin{proof}
We assume that $L(P+Q)=2$.
There are 13 classes of segments total, so $S$ and  $T$ have to share at least 5 segments.   
Let $S$ have lattice points $A$, $B$, $C$, $D$, and $E$, where $ABC$ is a triangle with sides of the same class and $D$ a point inside
that triangle.   If $AE$, $BE$, and $CE$ are all shared with $T$, then two of them share a vertex in $T$ and hence  one of the triangles  $ABE$, $ACE$, or $BCE$ is shared, which is impossible since sides of $ABC$ cannot be shared. Hence one of $AE$, $BE$, $CE$ is not shared. Similarly, one of $AD$, $BD$, $CD$ is not shared.
This implies that $ED$ is shared, and we can assume that $ADE$ is a shared triangle, $DB$ and $EC$ are shared, while $DC$ and $EB$ are not shared.

 
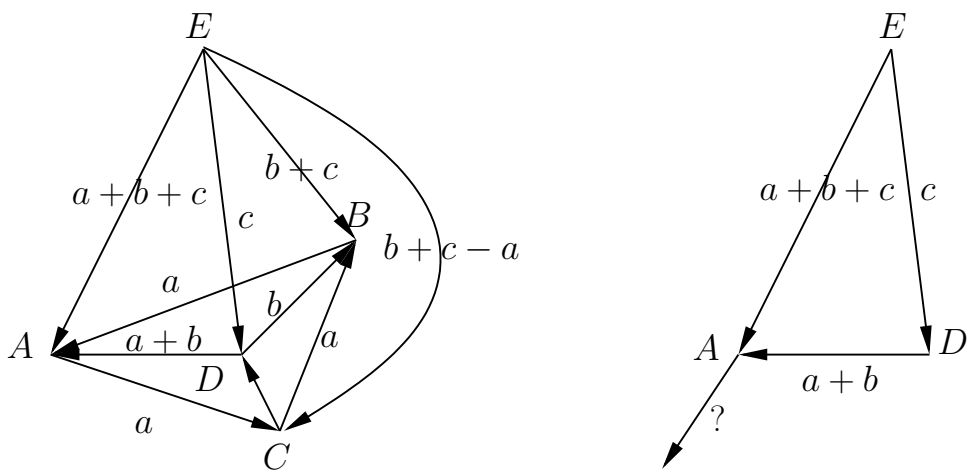
\begin{figure}[h]
\centerline{
 \scalebox{.8}
 {
\input{3plus7and10.pstex_t}}}
\end{figure}

Then $b$ and $b+c-a$ in $T$ should not be adjacent to $ED$ (or either $DC$ or $EB$ would be shared).
Hence one of $b$, $b+c-a$ has $A$ as an endpoint in $T$. If it is $b$, we have either $a$ or $a-b$ in $T$;
if it is $b+c-a$, then either $b+c$ or $a$ is in $T$. Both of these options are impossible. 
\end{proof}

\begin{proposition} Let $P$ and  $Q$ be 3D lattice polytopes of Minkowski length 1. If there exist 5-point  lattice subsets $S$ and $T$ of $P$ and $Q$
of types $2+2+(6)$ and  $(10)$ correspondingly,  then $L(P+Q)\geq 3$.
\end{proposition}
\begin{proof}
We assume that $L(P+Q)=2$.
There are 13 classes of segments total, so $S$ and  $T$ have to share at least 5 segments.    Let $S$ have lattice points $A$, $B$, $C$, $D$, and $E$, where $AB$ and $CD$ are from the same class, and  $BC$ and $AD$ are from the same class.
At least 3 segments starting at $E$ are shared. In $T$ at least two of them share a vertex, hence $S$ and $T$ share a triangle with vertex $E$. We can assume that that triangle is $AEC$. One of $EB$, $ED$ is also shared, let's assume it's $EB$.
This segment in $T$ cannot share a vertex with $AE$ or $EC$ (this would imply that either $AB$ or $BC$ is shared), so it connects two remaining vertices $X$ and $Y$.
If $ED$ is also shared, same would be true, but there are only five vertices, so $ED$ is not shared. Hence $BD$ is shared. But it would have  to have either $X$ or $Y$ as one of the vertices, hence $EBD$ is shared, a contradiction.
\end{proof}

\begin{theorem}\label{T:threeP} Let $P$, $Q$, $R$ be three 3D lattice polytopes of Minkowski length 1 with at least five lattice points each. Then $L(P+Q+R)\geq 4$.  
\end{theorem}
\begin{proof}
Let $S, T,$ and $U$ be any 5-point lattice subsets of $P$, $Q$, and $R$ correspondingly.  By the above propositions, $S$, $T$, and $U$ are all of type $(10)$.
If any of the three polytopes has more than five points, then there are at least 15 lattice segments connecting them. Hence there are multiple segments of the same class. Reducing the number of points, we  get  sets $S$, $T$, and $U$ where at least one of these lattice sets is not of type $(10)$.

Hence we can assume that each of $P$, $Q$, and $R$ has exactly five lattice points and  is of type $(10)$. Furthermore,  all three polytopes are translates of each other.
It remains to show that $L(3P)\geq4$.

Let us assume first that $P$ has four lattice vertices  $A, B, C, D$ and a lattice point $E$ inside.  Let $G$ be the  centroid of $P$. Draw through $G$ four planes, each parallel to one of the facets of $P$. Each of the planes cuts off a tetrahedron off of $P$, which is similar to $P$ with a coefficient of $3/4$. These four tetrahedra 
cover $P$, so the interior lattice  point $E$ belongs to at least one of them, say, to the tetrahedron one of whose vertices is $A$. Then  if we continue $AE$ to the point of intersection  $F$ with the plane $BCD$, then $AE/AF\leq 3/4$ and  $3AF\geq 4AE$, so in $3P$ we  have a segment $3AF$ whose lattice length   is at least 4. We have shown that in this case $L(3P)\geq 4$.

Next, let $P$ of type (10) have 5 lattice vertices and no other lattice points. It was shown in \cite{Scarf}, Theorem~3.5 that $P$ is $\A{3}$ equivalent to a polytope with the vertices $(0,0,0)$, $(1, 0, 0)$, $(0,1,0)$, $(0,0,1)$, and $(1, a, b)$ where $\gcd(a, b)=1$ and $0\leq a\leq b$. The sum of the two segments, $[0, a, b]$ and $[0,1,0]$ is a parallelogram of area $b$.
If $b\geq 3$ then by Lemma~1.7 of \cite{SS1} the Minkowski length of this parallelogram is at least $3$ and hence $L(2P)\geq 3$. We are left with two cases $a=b=1$ and $a=1$, $b=2$ (if $a=0$ then $L(P)\neq 1$). In the second case $P$ is not of type $(10)$ as $(0,a,b)=(0,1,-1)$.  It remains to deal with the case $a=b=1$.  Let $I$ be the vertical segment of length 1 and let $J$ be the segment connecting the origin to $(1,1,0)$. Then $I+1/2J\subseteq P$ and hence $2I+J\subseteq 2P$, so we again have $L(2P)\geq 3$.
\end{proof}

Notice that we have also proved the following Corollary.

\begin{corollary} If $P$ and $Q$ are 3D lattice polytopes of Minkowski length 1 with at least 5 lattice points each, then $L(P+Q)\geq 3$ unless $P$ and $Q$ are of type $(10)$ and are the same up to translation.
\end{corollary}

Next example demonstrates that one could have $L(2P)=2$ for a 3D lattice polytope $P$ of type $(10)$.
  
\begin{example}  Let $P$ have the vertices $(-1,-1,-1)$, $(1,0, 0)$, $(0,1,0)$, and $(0,0,1)$. Then  $P$ is of type (10), $L(P)=1$, and $L(2P)=2$.
\end{example}

\begin{proposition}\label{P:5and6} Let $P$ and $Q$ be 3D lattice polytopes of Minkowski length 1. If $P$ has at least 6 lattice points and $Q$ has at least 5, then $L(P+Q)\geq 3$. 
\end{proposition}
\begin{proof}
Since $P$ has $6$ lattice points, there are 15 lattice segments in $P$. Since there are 13 classes total, some of the segments will repeat and we will either have two segments
of the same class not sharing endpoints, or three segments of the same class forming a triangle. Hence we can pick $5$ points in $P$ that form a lattice subset  of type other than (10), which implies $L(P+Q)\geq 3$.
\end{proof}

\begin{theorem}\label{T:main} Let $Q_1+\cdots+Q_N\subseteq P$ be a maximal decomposition. Let $I(Q_i)$ denote the number of interior lattice points with respect to the 3D topology.
Then the overall number of interior lattice points
$$I=\sum_{i=1}^{N}I(Q_i)\leq 4.
$$
Furthermore, if more than one $Q_i$ has interior lattice points with respect to  the 3D topology, we have $I\leq 2$.
\end{theorem}
\begin{proof} 
In order to have an interior lattice point with respect to  the 3D topology, $Q_i$ has to have at least 5 lattice points. Hence, by Theorem \ref{T:threeP}, at most two of the $Q_i$'s could have interior lattice points and be three-dimensional.  Let's assume that this holds for two of  the $Q_i$'s.   If one of these two  $Q_i$'s has at least 2 interior lattice points, we get a contradiction with Proposition~\ref{P:5and6}. Hence in this case the total number of   interior lattice points is at most 2.  If only one of the $Q_i$'s  has interior lattice points with respect to  the 3D topology, then there are at most 4 of them, as $L(Q_i)=1$ and $Q_i$ has at most 8 lattice points total.
\end{proof}
The following example shows that there exists a Minkowski length one 3D polytope that has 4 interior lattice points, so the bound of Theorem~\ref{T:main} is sharp.

\begin{example} Consider a simplex $P$ with the vertices $(0,0,0)$,  $(1,3,0)$,  $(0,2,3)$, and $(4,1,3)$. Then  the interior lattice points of $P$  are  
$(1,2,1)$,  $(1,2,2)$,  $(1,1,1)$, and $(2,1,2)$. This can be checked by hand or using Polymake~\cite{Polymake}.  It is easy to verify that there are no parallel 
lattice segments connecting lattice points in $P$, which implies that  $L(P)=1$.
\end{example}

An example of a 3D lattice polytope of Minkowski length one with 8 lattice points (all of them on the boundary) was given in an MSRI-UP project directed by John Little \cite{MSRI}.  The number of lattice points in a lattice polytope in the $n$-space is at most $2^n$. A group of REU students at Kent  in Summer 2011 constructed  a lattice $n$-dimensional Minkowski length one polytope with $2^n$ points. It would be interesting to see  if there exists an $n$-dimensional simplex that has $2^n$ lattice points and Minkowski length one. 

\subsection*{Acknowledgments} 

We are thankful to the anonymous referee for pointing out  a gap in one of the arguments as well as for numerous corrections and suggestions.


\end{document}